\documentclass{article}

\usepackage{amsmath, amssymb}
\usepackage{graphicx, wrapfig} 
\usepackage{amsthm}
\usepackage{url}

\theoremstyle{plain}
\newtheorem{thm}{Theorem}
\newtheorem{lem}{Lemma}

\theoremstyle{definition}

\newtheorem{ex}{Example}

\newtheorem{prop}{Proposition}
\newtheorem{cor}{Corollary}

\title{Determinant of the OU matrix of a braid diagram}
\author{Ayaka Shimizu\thanks{Osaka Central Advanced Mathematical Institute, Osaka Metropolitan University, Sugimoto, Osaka, 558-8585, Japan. Email: shimizu1984@gmail.com}
and Yoshiro Yaguchi\thanks{Maebashi Institute of Technology, Maebashi, Gunma, 371-0816, Japan. Email: y.yaguchi@maebashi-it.ac.jp}}
\date{\today}

\begin{document}

\maketitle

\begin{abstract}
In this paper, we define the OU matrix of a braid diagram and discuss how the OU matrix reflects the warping degree or the layeredness of the braid diagram, and show that the determinant of the OU matrix of a layered braid diagram is the product of the determinants of the layers. 
We also introduce invariants of positive braids which are derived from the OU matrix. 
\end{abstract}

\section{Introduction}
\label{section-intro}

A {\it braid} is a disjoint union of strands attached to two parallel horizontal bars in $\mathbb{R}^3$, where each strand runs from one bar to the other bar without returning. 
We say a braid is an {\it $n$-braid} when the braid is consisting of $n$ strands. 
A {\it braid diagram} is a diagram of a braid depicted on $\mathbb{R}^2$. 
Each braid diagram is represented by a word consisting of $\sigma_i$ and ${\sigma_i}^{-1}$, where $\sigma_i$ (resp. ${\sigma_i}^{-1}$) stands for a positive (resp. negative) crossing, as shown in Figure \ref{fig-braid}.
\begin{figure}[ht]
\centering
\includegraphics[width=6cm]{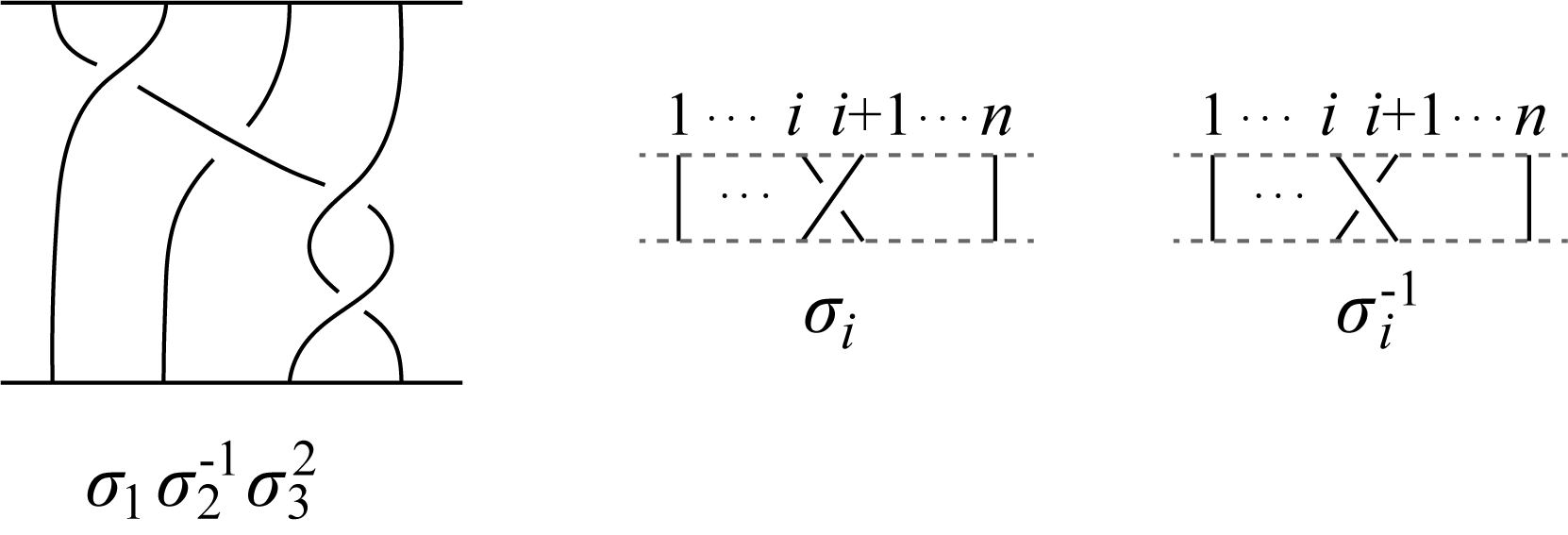}
\caption{A braid diagram. }
\label{fig-braid}
\end{figure}

The warping degree has been defined for link diagrams and studied to measure a complexity of link diagrams in a sense of over/under information (see, for example, \cite{K-l, K-w, ASl}). 
Considering the order of strands, the warping degree of a braid diagram\footnote{In \cite{NSP}, a labeling named ``warping labeling'' was also defined for twisted virtual braid diagrams from a different point of view of the warping degree, following a rule of the warping degree of the closure.} was also defined recently in \cite{ASA} as follows.
Let $s_1 , s_2 , \dots , s_n$ be the strands of an $n$-braid diagram $B$ which are positioned in the order from the left-hand side to the right-hand side at the top of the braid diagram, as shown in Figure \ref{fig-5}. 
The {\it warping degree of $(s_i , s_j)$}, denoted by $\mathrm{wd}(s_i , s_j)$, is the number of crossings such that $s_i$ is under $s_j$. 
Let $\mathbf{s}$ be a sequence of the strands in some order, and we denote the pair of $B$ and $\mathbf{s}$ by $B_{\mathbf{s}}$. 
The {\it warping degree of $B_{\mathbf{s}}$}, denoted by $\mathrm{wd}(B_{\mathbf{s}})$, is the total number of crossings such that $s_i$ is under $s_j$ where $s_i$ is positioned ahead of $s_j$ in the sequence $\mathbf{s}$. 
The {\it warping degree of a braid diagram $B$}, denoted by $\mathrm{wd}(B)$, is the minimum value of $\mathrm{wd}(B_{\mathbf{s}})$ for all $\mathbf{s}$. 

\begin{ex}
The braid diagram $B$ in Figure \ref{fig-5} has $\mathrm{wd}(s_1 , s_2)=1$, $\mathrm{wd}(s_2 , s_1)=0$. 
For the sequences of strands $\mathbf{s}=(s_1 , s_2 , s_3 , s_4 , s_5 )$ and $\mathbf{s}'=(s_2 , s_3 , s_1 , s_4 , s_5 )$, $B$ has $\mathrm{wd}(B_{\mathbf{s}})=3$ and $\mathrm{wd}(B_{\mathbf{s}'})=1$. 
We can determine the warping degree to be $\mathrm{wd}(B)=1$ by checking the warping degrees for all the $5!$ orders or applying Theorem \ref{thm-det-wd}.
\begin{figure}[ht]
\centering
\includegraphics[width=5cm]{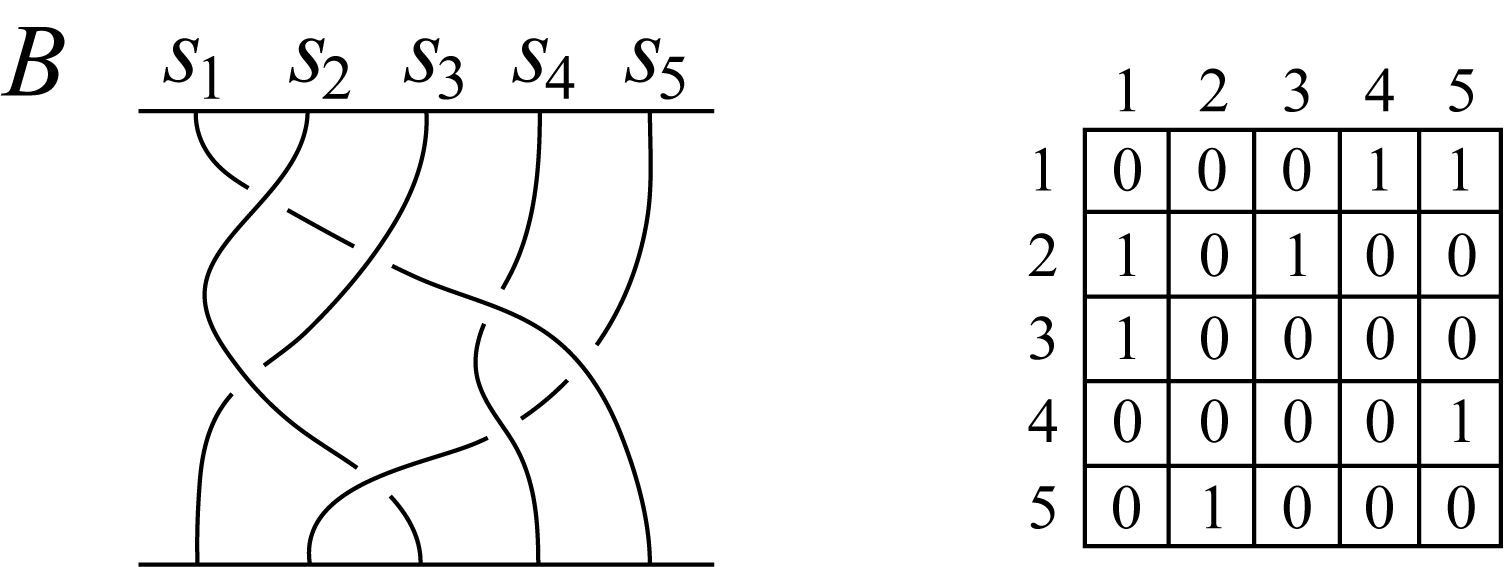}
\caption{A braid diagram $B$ and its OU matrix with $\mathbf{s} = (s_1, s_2 , s_3 , s_4 , s_5)$. }
\label{fig-5}
\end{figure}
\label{ex-5}
\end{ex}

When the set of strands $S= \{ s_1, s_2, \dots , s_n \}$ of an $n$-braid diagram $B$ can be divided as the disjoint union $S_1 \amalg S_2 \amalg \dots \amalg S_k$ of $k$ non-empty sets $S_1 , S_2 , \dots , S_k$ for some integer $k \geq 2$ so that $\mathrm{wd}(s_i, s_j)=0$ for any $s_i \in S_l$ and $s_j \in S_m$ for $l < m$, we say $B$ is {\it layered}. 
For each $i \in \{ 1, 2, \dots , k \}$, we assume the union of strands in $B$ which belong to $S_i$ as a $\# (S_i)$-braid diagram, regardless of the position of the endpoints. 
For example, we assume the diagrams shown in the center and the right-hand side in Figure \ref{fig-layer} as a $3$-braid and $2$-braid diagrams, respectively. 
We call each braid diagram $B_i$ consisting of the strands in $S_i$ a {\it layer} and denote $B= B_1 \oplus B_2 \oplus \dots \oplus B_k$. 
In particular, we say an $n$-braid diagram $B$ is {\it completely layered} when $B$ is layered with $n$ layers where each layer is consisting of a single strand. 
By definition, $B$ is completely layered if and only if $\mathrm{wd}(B)=0$. 
Therefore, the warping degree represents how far a braid diagram $B$ is from a completely layered braid diagram, like the warping degree of a link diagram $D$ represents how far $D$ is from a monotone diagram, which is a diagram of a trivial link. 

\begin{figure}[ht]
\centering
\includegraphics[width=7cm]{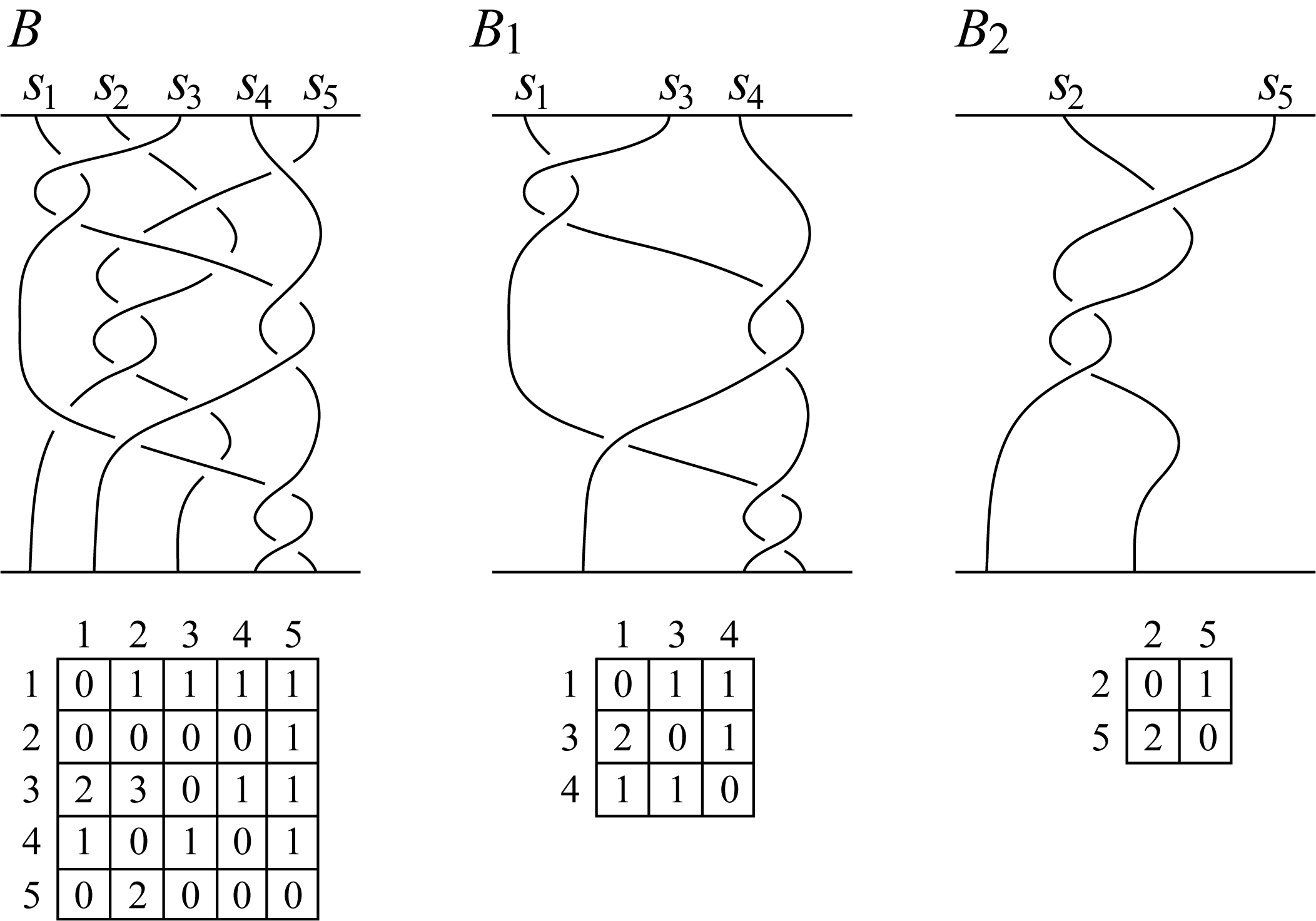}
\caption{A layered diagram $B= B_1 \oplus B_2$ with the layers $B_1$ and $B_2$ of the sets of strands $S_1= \{ s_1, s_3 , s_4 \}$ and $S_2 = \{ s_2 , s_5 \}$. The determinants are $\mathrm{det}(B)=-6$, $\mathrm{det}(B_1)=3$ and $\mathrm{det}(B_2 )=-2$, and we have $\mathrm{det}(B)= \mathrm{det}(B_1) \mathrm{det}(B_2)$.  }
\label{fig-layer}
\end{figure}

The warping degree of a braid diagram is useful to measure a complexity of the closure as well. 
For a closure $D$ of a pure braid $B$, we have the inequality\footnote{The non-pure version of the inequality was also shown in \cite{ASA} with some condition of the sequence of strands.} $u(D) \leq \mathrm{wd}(B)$ shown in \cite{ASA}, namely the warping degree is an upper bound of the unlinking number $u(D)$ of the link diagram $D$. 
Even for non-pure braids, the warping degree was employed to estimate the unknotting number of the closure for weaving knots in \cite{ASA}. 
However, finding the warping degree is not always easy for large number of strands; each $n$-braid has $n!$ orders of strands. 
We call the problem to find the value of the warping degree the {\it warping degree problem}. 

In this paper, we will define the ``over-under matrix,'' or OU matrix, which represents the over/under information of each pair of strands of a braid diagram. 
By definition given in Section \ref{subs-def}, the OU matrix $M(B_{\mathbf{s}})$ of a braid diagram $B$ with the sequence of strands $\mathbf{s}=(s_1, s_2, \dots , s_n )$ is the matrix $(m_{ij})$ such that $m_{ij}= \mathrm{wd}(s_j , s_i)$ (see Figure \ref{fig-5}).
We will show that the determinant of the OU matrix $M(B_{\mathbf{s}})$ does not depend on the order of the strands of $\mathbf{s}$ in Section \ref{section-oum} and denote it by $\mathrm{det}(B)$, and discuss how the determinant reflects the warping degree or layeredness of braid diagrams, showing the following theorems (see Figure \ref{fig-layer}). 

\medskip
\begin{thm}
Let $B$ be a layered braid diagram with layers $B_1$, $B_2$. 
The determinant of the OU matrix of $B$ is the product of the determinants of $B_1$ and $B_2$, namely, 
$$\mathrm{det}(B_1 \oplus B_2)=\mathrm{det}(B_1)\mathrm{det}(B_2).$$
\label{thm-det-layer}
\end{thm}

\begin{thm}
For any braid diagram $B$, we have $\mathrm{wd}(B)\neq 0$ if $\mathrm{det}(B) \neq 0$. 
\label{thm-det-wd}
\end{thm}
\medskip

\noindent The rest of the paper is organized as follows. 
In Section \ref{section-oum}, we define and study the OU matrix and prove Theorem \ref{thm-det-layer}. 
In Section \ref{section-wd-cm}, we discuss the warping degree problem in terms of the ``linear ordering problem'' on the OU matrices, and prove Theorem \ref{thm-det-wd}. 
In Section \ref{section-pb}, diagram-independent invariants of positive braids are introduced as an application with examples.

\section{Over-under matrix}
\label{section-oum}

In this section, we define the OU matrix in Subsection \ref{subs-def}, and investigate its properties in Subsections \ref{subs-sim} and \ref{subs-prop}. 
Then we prove Theorem \ref{thm-det-layer} in Subsection \ref{subs-pf}

\subsection{Definition of the OU matrix}
\label{subs-def}

In this paper, we deal with two types of permutations, the ``braid permutation'' $\rho$ and the ``strand permutation'' $\pi$. 
Since each $n$-braid connects the points on the top and bottom, there is a certain permutation $\rho : \{ 1, 2, \dots , n \} \to \{ 1, 2, \dots , n \}$ for each braid or braid diagram. 
When a strand of a braid $\beta$ or a braid diagram $B$ is at the $i$th position from the left-hand side on the top and at the $j$th on the bottom, we denote $\rho (i)=j$. 
We call such $\rho$ the {\it braid permutation} associated to $\beta$ or $B$. 
For example, the braid $\sigma_1 \sigma_2^{-1}\sigma_3^2$ in Figure \ref{fig-braid} has braid permutation $\rho (1,2,3,4)=(3,1,2,4)$. 
Braid permutation is one of the basic tools to apart braids. 
In this paper, we consider one more permutation to discuss a complexity of braid diagrams. 
For an $n$-braid diagram $B$, label the strands $s_1, s_2, \dots ,s_n$ from the left-hand side to the right-hand side on the top of the braid diagram. 
We can consider the permutation on $( s_1, s_2, \dots , s_n )$ to be the order of strands. 
Let $\pi : \{ 1, 2, \dots , n \} \to \{ 1, 2, \dots , n \}$ be a permutation, where $\pi (i) = j$ implies that the $i$th component of $\pi$ is $s_j$. 
For example, when $\pi (1)=3$, $\pi(2)=2$, $\pi(3)=1$ and $\pi(4)=4$ (we also denote that by $\pi (1, 2, 3, 4)=(3, 2, 1, 4)$, or simply $\pi = (3, 2, 1, 4)$), we have the sequence $\mathbf{s}' =( s_3 , s_2 , s_1 , s_4 )$ associated to $\pi$. 
We call such a permutation a {\it strand permutation}, or simply a permutation in this paper.

Let $B$ be a braid diagram with $n$ strands. 
Label the strands $s_1, s_2, \dots ,s_n$ from the left to the right on the top of the braid diagram. 
Let $\mathbf{s}$ be a sequence of the strands $s_1, s_2, \dots ,s_n$ in any order. 
The {\it Over-under matrix}, or OU matrix, denoted by $M(B_{\mathbf{s}})$, is an $n \times n$ matrix such that the components on the main diagonal are all zero and each $(i,j)$-component represents the number of over-crossings on the $i$th strand over the $j$th strand, where the ``$k$th strand'' stands for the $k$th component of the sequence $\mathbf{s}$. 
As mentioned in Section \ref{section-intro}, $M(B_{\mathbf{s}})$ has the $(i,j)$-component as the warping degree $\mathrm{wd}(s_j , s_i)$ when $\mathbf{s} =(s_1 , s_2 , \dots , s_n)$. 
Since the matrix $M(B_{\mathbf{s}})$ is determined by the strand permutation $\pi$ on $( s_1, s_2, \dots , s_n )$, we also denote $M(B_{\mathbf{s}})$ by $M_B(\pi)$.

\subsection{Permutation-independent quantities}
\label{subs-sim}

In this subsection, we show that the OU matrices of the same braid diagram with any strand permutation are similar to each other, and consider some quantities derived from the OU matrix that are permutation-independent. 
Let $I_{kl}$ be an elementary matrix which is obtained from the $n \times n$ identity matrix $I$ by swapping the $k$th and $l$th rows, namely the square matrix $(m_{ij})$ such that $m_{kl}=m_{lk}=1$, $m_{ii}=1$ if $i \neq k, l$, and $m_{ij}=0$ otherwise. 
For transposition, we have the following lemma. 

\medskip
\begin{lem}
Let $\pi$, ${\pi}'$ be permutations on $( 1, 2, \dots , n )$ which are related by a single transposition $\tau_{kl}$. 
Let $M_B(\pi)$, $M_B({\pi}')$ be the OU matrices for the same $n$-braid diagram $B$. 
We have
$$ M_B ({\pi}')= I_{kl} M_B (\pi) I_{kl}.$$
\label{lem-I}
\end{lem}

\begin{proof}
By definition, $M_B ({\pi}')$ is obtained from $M_B (\pi)$ by swapping the $k$th and $l$th rows and then the $k$th and $l$th columns (or in the opposite order). 
\end{proof}
\medskip

\begin{ex}
Let $id =(1,2,3,4)$, $\pi = (3,2,1,4)$. 
Since they are related by a transposition $\tau_{13}$, we have $M_B ( \pi )=I_{13} M_B ( id ) I_{13}$ for the braid diagram $B$ in Figure \ref{fig-3214}, namely, 
\begin{align*}
\begin{pmatrix}
0 & 2 & 0 & 1 \\
0 & 0 & 2 & 1 \\
0 & 1 & 0 & 0 \\
0 & 1 & 0 & 0 
\end{pmatrix}
=
\begin{pmatrix}
0 & 0 & 1 & 0 \\
0 & 1 & 0 & 0 \\
1 & 0 & 0 & 0 \\
0 & 0 & 0 & 1 
\end{pmatrix}
\begin{pmatrix}
0 & 1 & 0 & 0 \\
2 & 0 & 0 & 1 \\
0 & 2 & 0 & 1 \\
0 & 1 & 0 & 0 
\end{pmatrix}
\begin{pmatrix}
0 & 0 & 1 & 0 \\
0 & 1 & 0 & 0 \\
1 & 0 & 0 & 0 \\
0 & 0 & 0 & 1 
\end{pmatrix} .
\end{align*}
We also have $M_B(\pi)=I_{12} I_{23} I_{12} M_B(id) I_{12} I_{23} I_{12}$.
\begin{figure}[ht]
\centering
\includegraphics[width=6cm]{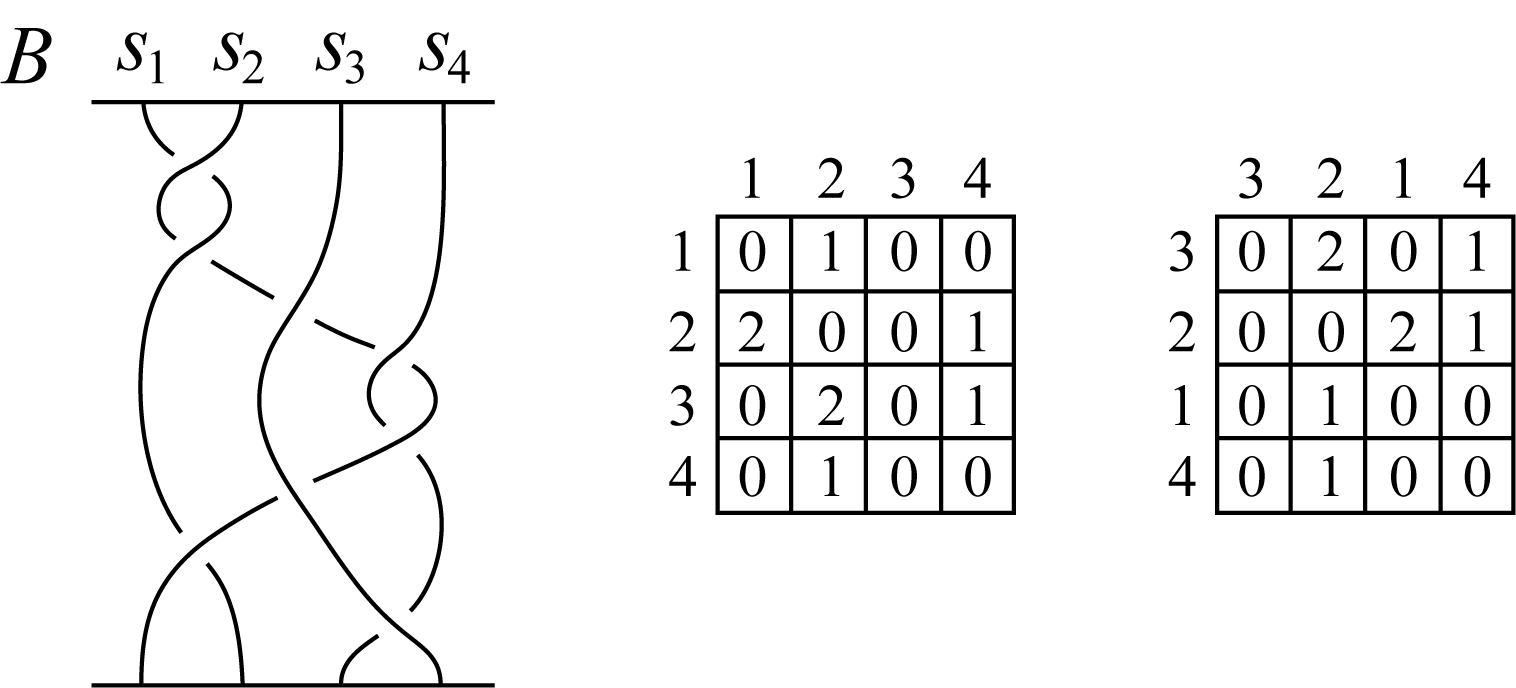}
\caption{A braid diagram $B$ and the OU matrices $M(B_{\mathbf{s}})$ and $M(B_{\mathbf{s}'})$ for  $\mathbf{s}=( s_1, s_2 , s_3 , s_4 )$ and $\mathbf{s}' = ( s_3 , s_2 , s_1 , s_4 )$. }
\label{fig-3214}
\end{figure}

\end{ex}
\medskip

\noindent Since any pair of permutations are related by a finite sequence of transpositions, we have the following proposition. 

\medskip
\begin{prop}
For any braid diagram $B$ and permutations $\pi$ and ${\pi}'$, $M_B ({\pi}')$ is similar to $M_B (\pi)$. Namely, there is an invertible matrix $N$ such that 
$$M_B ({\pi}') = N^{-1} M_B (\pi ) N.$$
\label{prop-similar}
\end{prop}

\begin{proof}
When ${\pi}'$ is obtained from $\pi$ by a sequence of transpositions $\tau_{k_1 l_1}$, $\tau_{k_2 l_2}$, $\dots$ , $\tau_{k_i l_i}$, we have 
\begin{align*}
M_B ({\pi}')= I_{k_i l_i} \dots I_{k_2 l_2}I_{k_1 l_1} M_B (\pi) I_{k_1 l_1} I_{k_2 l_2} \dots I_{k_i l_i}
\end{align*}
by Lemma \ref{lem-I}. 
Since $\left( I_{k_i l_i} \dots I_{k_2 l_2}I_{k_1 l_1} \right) \left( I_{k_1 l_1} I_{k_2 l_2} \dots I_{k_i l_i} \right) =I$, the matrix $N= I_{k_1 l_1} I_{k_2 l_2} \dots I_{k_i l_i}$ satisfies $M_B ({\pi}') = N^{-1} M_B (\pi ) N$.
\end{proof}
\medskip

\noindent From Proposition \ref{prop-similar}, we have the following corollary. 

\medskip
\begin{cor}
For any braid diagram $B$, the rank, determinant, trace, eigenvalues and characteristic polynomial of the OU matrix $M_B( \pi )$ do not depend on the permutation $\pi$.
\label{cor-det}
\end{cor}
\medskip

\noindent From now on, we denote $\mathrm{rank}M(B_{\mathbf{s}})$ or $\mathrm{rank}M_B (\pi)$ by $\mathrm{rank}(B)$. 
We also denote $\mathrm{det}M(B_{\mathbf{s}})$ or $\mathrm{det}M_B (\pi)$ by $\mathrm{det}(B)$. 
Moreover, we have the following corollary. 

\medskip
\begin{cor}
For an OU matrix $(m_{ij})$ of a braid diagram $B$, the multiset of the $n$ multisets 
$$O(B)= \{ \{ m_{11}, m_{12}, \dots , m_{1n} \} , \{ m_{21}, m_{22}, \dots , m_{2n} \} , \dots , \{ m_{n1}, m_{n2}, \dots , m_{nn} \} \}$$ 
does not depend on the strand permutation $\pi$. 
The multiset 
$$U(B)= \{ \{ m_{11}, m_{21}, \dots , m_{n1} \} , \{ m_{12}, m_{22}, \dots , m_{n2} \} , \dots , \{ m_{1n}, m_{2n}, \dots , m_{nn} \} \}$$ 
does not depend on $\pi$, too. 
\label{prop-set}
\end{cor}
\medskip

\noindent We note that the multiset $\{ m_{i1}, m_{i2}, \dots , m_{in} \}$ has the components $m_{ii}=0$ and $m_{ij}$ as the number of over-crossings on $s_i$ over $s_j$. 
We call the multiset $O(B)$ the {\it over-crossing set} and $U(B)$ the {\it under-crossing set of $B$}. 

\medskip
\begin{ex}
The braid diagram $B$ in Figure \ref{fig-3214} has $\mathrm{rank}(B)=3$, $\mathrm{det}(B)=0$ and the eigenvalues $0, \pm \sqrt{3}$. 
The over- and under-crossing sets are
\begin{align*}
O(B)= \{ \{ 0,0,0,1 \} , \{ 0,0,0,1 \} , \{ 0,0,1,2 \} , \{ 0,0,1,2 \} \} , \\
U(B)= \{ \{ 0,0,0,0 \} , \{ 0,0,0,2 \} , \{ 0,0,1,1 \} , \{ 0,1,1,2 \} \} . 
\end{align*}
\end{ex}
\medskip

\subsection{Properties of OU matrix}
\label{subs-prop}

In this subsection, we see properties of OU matrix. 
For the reverse ${\pi}'=( \pi (n), \pi (n-1), \dots , \pi (1))$ of a strand permutation $\pi = ( \pi (1), \pi (2) , \dots , \pi (n))$, we have the following lemma.

\medskip 
\begin{lem}
Let $\pi$ be a permutation on $(1,2, \dots , n)$ and let ${\pi}'$ be the reverse of $\pi$. 
Then, $M_B({\pi}') = I' M_B(\pi) I'$ holds, where $I'$ is the $n \times n$ matrix in which the $(i, n+1-i)$-component is $1$ for $i=1, 2, \dots , n$ and the others are $0$. 
\label{lem-pi-inv}
\end{lem}
\medskip 

\noindent In other words, $M_B({\pi}')$ is a matrix obtained from $M_B(\pi)$ by a clockwise (or counterclockwise) rotation by $180^{\circ}$, namely, the $(i,j)$-component of $M_B({\pi}')$ coincides with the $(n+1-i, n+1-j)$-component of $M_B(\pi)$.  

\medskip 
\begin{ex}
For the braid diagram $B$ shown in Figure \ref{fig-3214}, we have 
\begin{align*}
M_B(\pi)=
\begin{pmatrix}
0 & 1 & 0 & 0 \\
2 & 0 & 0 & 1 \\
0 & 2 & 0 & 1 \\
0 & 1 & 0 & 0 
\end{pmatrix}
, \ M_B({\pi}')=
\begin{pmatrix}
0 & 0 & 1 & 0 \\
1 & 0 & 2 & 0 \\
1 & 0 & 0 & 2 \\
0 & 0 & 1 & 0 
\end{pmatrix}
\end{align*}
when $\pi = (1,2,3,4)$ and ${\pi}' =(4,3,2,1)$. 
\end{ex}
\medskip 

\noindent {\it Proof of Lemma \ref{lem-pi-inv}.} \ 
On $M_B(\pi)$, each $(i,j)$-component represents the number of crossings where $s_{\pi (i)}$ is over $s_{\pi (j)}$, which appears at the $(n+1-i, n+1-j)$-component in $M_B({\pi}')$ since the $i$th component of $\pi$ coincides with the $(n+1-i)$th component of ${\pi}'$. 
\qed \\
\medskip

\noindent For braid diagrams $B$, $C$ with the same number of strands, we represent their braid product by $BC$. 
The following proposition will be used in Section \ref{section-pb}. 

\medskip
\begin{prop}
Let $B$, $C$ be braid diagrams of $n$ strands, and let $\rho_B$ be the braid permutation of $B$. We have 
$$M_{BC}(\pi)=M_B(\pi)+M_C(\pi \rho_B)$$
for any $\pi$. 
In particular, when $\rho_B = id$, i.e., $B$ is a pure braid, we have 
$$M_{BC}(\pi)= M_B(\pi)+M_C(\pi).$$
\label{prop-product}
\end{prop}

\begin{proof}
The $\pi (i)$th strand of $BC$ is the $\pi (i)$th strand in $B$, and the $\pi \rho_B (i)$th strand in $C$. 
Then the $(i,j)$-component of $M_{BC}(\pi)$, which is the number of crossings where $s_{\pi (i)}$ is over $s_{\pi (j)}$, is the sum of the number of crossings where the $\pi (i)$th strand is over the $\pi (j)$th strand in $B$ and the crossings where the $\pi \rho_B (i)$th strand is over the $\pi \rho_B (j)$th strand in $C$, which is the sum of the $(i,j)$-components of $M_B(\pi)$ and $M_C( \pi \rho_B)$.
\end{proof}

\subsection{Proof of Theorem \ref{thm-det-layer}}
\label{subs-pf}

In this subsection, we prove Theorem \ref{thm-det-layer}. 

\medskip
\begin{lem}
Let $B= B_1 \oplus B_2$ be a layered braid diagram with two layers $B_1$ and $B_2$ of the sets of strands $S_1$ and $S_2$, respectively. 
Take a sequence of strands $\mathbf{s}$ of $B$ so that any strand in $S_1$ is positioned ahead of any strand in $S_2$. 
Then, the OU matrix $M(B_{\mathbf{s}})$ consists of the four blocks as 
\begin{align*}
M(B_{\mathbf{s}}) =
\begin{pmatrix}
M_1 & N \\
O & M_2
\end{pmatrix},
\end{align*}
where $M_i$ is the OU matrix of $B_i$ with the sequence of strands of $B_i$ following the order of $\mathbf{s}$ for $i=1,2$, and $O$ is a zero matrix.
\label{lem-block}
\end{lem}

\begin{proof}
Let $\mathbf{s}=( s^1_1 , s^1_2 , \dots , s^1_{k_1} , s^2_1 , s^2_2 , \dots , s^2_{k_2})$ be a sequence of strands of $B$, where $s^1_1 , s^1_2 , \dots , s^1_{k_1} \in S_1$ and $s^2_1 , s^2_2 , \dots , s^2_{k_2} \in S_2$. 
The block $M_i$ represents the over/under relation between $s^i_1 , s^i_2 , \dots , s^i_{k_i}$ for $i=1,2$. 
The lower left block represents the number of crossings such that $s_m$ is over $s_l$ for all $s_m \in S_2$ and $s_l \in S_1$, which are all zero for $B= B_1 \oplus B_2$.  
\end{proof}
\medskip 

\noindent We prove Theorem \ref{thm-det-layer}. \\

\medskip 
\noindent {\it Proof of Theorem \ref{thm-det-layer}}. \ 
Let $B= B_1 \oplus B_2$ be a layered braid diagram with two layers $B_1$ and $B_2$ of the set of strands $S_1$ and $S_2$, respectively. 
Take a sequence of strands $\mathbf{s}$ same to Lemma \ref{lem-block}. 
By Lemma \ref{lem-block} and the well-known formula 
\begin{align*}
\mathrm{det}
\begin{pmatrix}
S & T \\
O & U 
\end{pmatrix}
=\mathrm{det} (S) \mathrm{det} (U),
\end{align*}
we have $\mathrm{det}(B)=\mathrm{det}(M(B_{\mathbf{s}}))= \mathrm{det}(M_1)\mathrm{det}(M_2)=\mathrm{det}(B_1)\mathrm{det}(B_2)$. \qed \\

\medskip 
\begin{cor}
When a braid diagram $B$ has a layer which consists of a single strand, $\mathrm{det}(B)=0$. 
\end{cor}

\section{Warping degree and OU matrix}
\label{section-wd-cm}

In Subsection \ref{subs-LOP}, we will discuss the warping degree problem in terms of the linear ordering problem and prove Theorem \ref{thm-det-wd}. 
In Subsection \ref{subs-weaving}, we will discuss the warping degree of weaving braids. 

\subsection{Linear ordering problem}
\label{subs-LOP}

For an $n \times n$ matrix $M$, the {\it linear ordering problem} is a permutation-based combinatorial optimization problem to find a simultaneous permutation of the rows and columns such that the sum of all the components below the main diagonal is minimized. 
The linear ordering problem has been studied for decades (\cite{CW}) in various fields, such as archeology, economics, sports tournaments, etc. 
For example, Leontief's input-output analysis (\cite{WL}) employs the ``input-output matrix,'' or IO matrix\footnote{The OU matrix is named after the IO matrix.}, to discuss bidirectional mutual relations between some groups. 
It was shown by Garey and Johnson that the linear ordering problem is an NP-hard problem (\cite{GJ}), and lots of strategies and effective algorithms to find optimal solutions have been studied.

In this subsection, we describe the warping degree of a braid diagram in terms of the linear ordering problem on OU matrices. 
Let $B$ be a braid diagram of $n$ strands, and $\mathbf{s}$ be the sequence of strands in the order associated to a permutation $\pi$. 
For the OU matrix $M(B_{\mathbf{s}}) = M_B ( \pi ) = (m_{ij})$, we define the objective function $f_B (\pi)$ to be the sum of all the components below the main diagonal, namely, 
$$f_B(\pi) = \sum_{1 \leq j \leq i \leq n} m_{ij}.$$

\medskip
\begin{prop}
For any braid diagram $B$ and permutation $\pi$, we have 
$$f_B(\pi) = \mathrm{wd}(B_{\mathbf{s}}) ,$$
where $\mathbf{s}$ is the sequence of strands of $B$ associated to $\pi$. 
\label{prop-f-wd}
\end{prop}
\medskip

\begin{proof}
By definition, the objective function $f_B(\pi)$ represents the number of crossings such that the former strand is under and equivalently the latter strand is over according to the order of the strand permutation $\pi$, and it is accordant to the warping degree. 
\end{proof}
\medskip

\noindent From Proposition \ref{prop-f-wd}, we have 
$$\mathrm{wd}(B)= \min_{\pi} f_B(\pi).$$
Hence finding the value of $\mathrm{wd}(B)$ is equivalent to finding an optimal solution for the linear ordering problem on the OU matrix for given $B$. 
We prove Theorem \ref{thm-det-wd}, namely we show that $\mathrm{wd}(B) \neq 0$ if $\mathrm{det}(B) \neq 0$ for any braid diagram $B$. \\

\medskip
\noindent {\it Proof of Theorem \ref{thm-det-wd}.} \ 
When $\mathrm{wd}(B)=0$, we have $f_B( \pi )=0$ for some permutation $\pi$. 
Then $\mathrm{det}(B)=\mathrm{det}(M_B(\pi))=0$ since $M
_B(\pi)$ is a strictly upper triangular matrix. 
Hence, $\mathrm{wd}(B) \neq 0$ when $\mathrm{det}(B) \neq 0$
\qed \\
\medskip

\begin{ex}
Recall Example \ref{ex-5}. 
The braid diagram $B$ in Figure \ref{fig-5} has $\mathrm{det}(B)=1$. 
Therefore, $\mathrm{wd}(B) \neq 0$. 
\end{ex}
\medskip

\subsection{Weaving braids}
\label{subs-weaving}

A {\it weaving link} $W(p,q)$ is a closure of the braid 
$$B_W(p,q)= \left( \sigma_1 \sigma_2^{-1} \dots \sigma_{p-1}^{(-1)^p} \right)^q .$$
In this subsection, we determine the warping degree of the canonical diagram of $B_W(p,p)$, improving a result in \cite{ASA}. 
The following example was shown in \cite{ASA} to show that the warping degree depends on the order of strands. 

\medskip
\begin{ex}[\cite{ASA}]
The braid diagram $B=B_W(7,7)$ with strands $s_1, s_2 , \dots , s_7$ depicted in Figure \ref{fig-w77} has $\mathrm{wd}(B_{\mathbf{s}})=22$ and $\mathrm{wd}(B_{\mathbf{s}'})=12$ for $\mathbf{s}=(s_1 , s_2 , s_3 , s_4 , s_5 , s_6 , s_7)$ and $\mathbf{s}' =(s_1 , s_3 , s_5 , s_7 , s_2 , s_4 , s_6 )$.
\begin{figure}[ht]
\centering
\includegraphics[width=2.5cm]{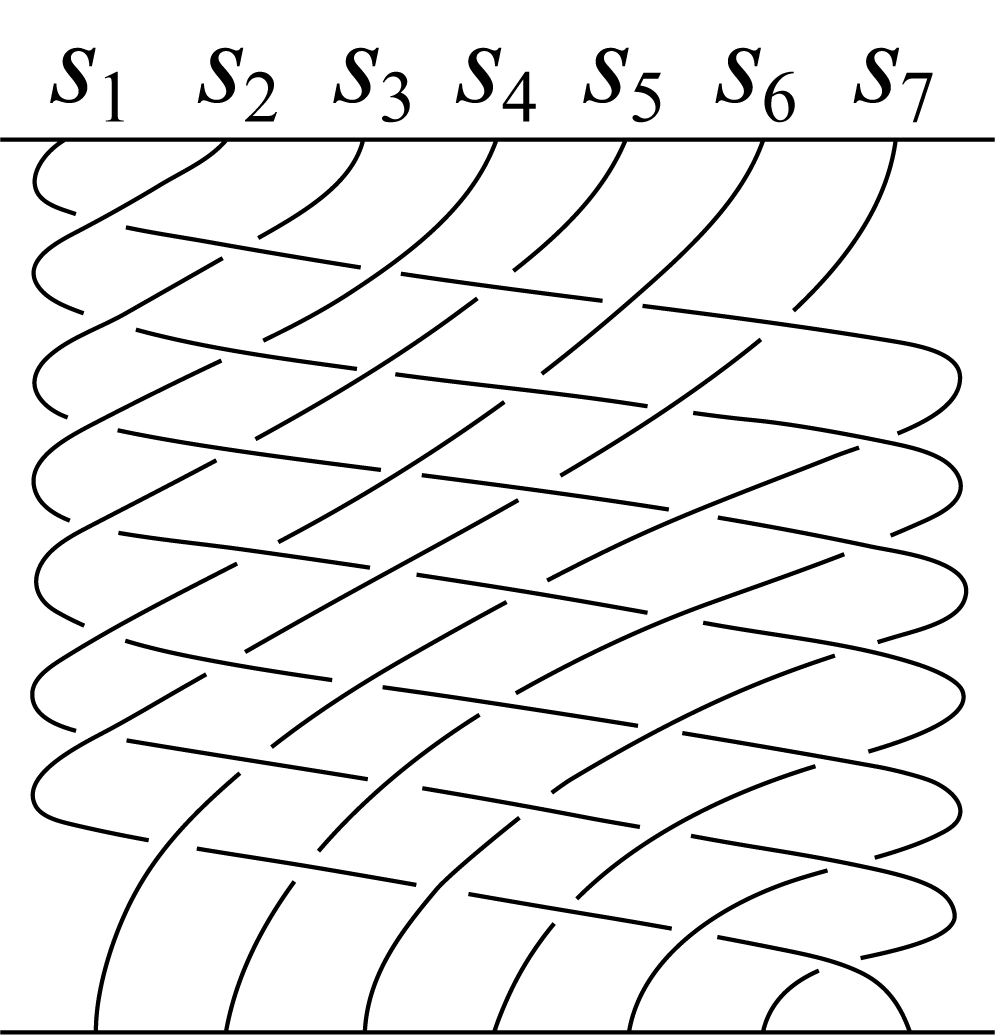}
\caption{$B_W(7,7)$.}
\label{fig-w77}
\end{figure}
\label{ex-w77}
\end{ex}
\medskip

\noindent In this section, we denote the canonical diagram of the braid $B_W(p,q)$ by the same notation $B_W(p,q)$. 
The following proposition was proved in \cite{ASA} to estimate the unknotting number of weaving knots.

\medskip
\begin{prop}[\cite{ASA}]
When $p$ is an odd number, the inequality $\mathrm{wd}(B_W(p,p))\leq \frac{p^2 -1}{4}$ holds. 
When $p$ is an even number, the equality $\mathrm{wd}(B_W(p,p)) = \frac{p(p-1)}{2}$ holds. 
\label{prop-w-in}
\end{prop}
\medskip

\noindent We will improve the proposition considering the OU matrices, which have regular patterns for some certain sequences of strands as shown in the following example.

\medskip
\begin{ex}
The OU matrices of the braid diagram $B=B_W(7,7)$ in Example \ref{ex-w77} with $\mathbf{s}=(s_1 , s_2 , s_3 , s_4 , s_5 , s_6 , s_7)$ and $\mathbf{s}' =(s_1 , s_3 , s_5 , s_7 , s_2 , s_4 , s_6 )$ are
\begin{align*}
M(B_{\mathbf{s}})= 
\begin{pmatrix}
{\bf 0} & 0 & 2 & 0 & 2 & 0 & 2 \\
2 & {\bf 0} & 0 & 2 & 0 & 2 & 0 \\
0 & 2 & {\bf 0} & 0 & 2 & 0 & 2 \\
2 & 0 & 2 & {\bf 0} & 0 & 2 & 0 \\
0 & 2 & 0 & 2 & {\bf 0} & 0 & 2 \\
2 & 0 & 2 & 0 & 2 & {\bf 0} & 0 \\
0 & 2 & 0 & 2 & 0 & 2 & {\bf 0}
\end{pmatrix}
, \ M(B_{\mathbf{s}'})= 
\begin{pmatrix}
{\bf 0} & 2 & 2 & 2 & 0 & 0 & 0 \\
0 & {\bf 0} & 2 & 2 & 2 & 0 & 0 \\
0 & 0 & {\bf 0} & 2 & 2 & 2 & 0 \\
0 & 0 & 0 & {\bf 0} & 2 & 2 & 2 \\
2 & 0 & 0 & 0 & {\bf 0} & 2 & 2 \\
2 & 2 & 0 & 0 & 0 & {\bf 0} & 2 \\
2 & 2 & 2 & 0 & 0 & 0 & {\bf 0} 
\end{pmatrix}.
\end{align*}
\label{ex-w77-m}
\end{ex}
\medskip

\noindent The following proposition is an improvement of Proposition \ref{prop-w-in}. 

\medskip
\begin{prop}
When $p$ is an odd number, the equality $\mathrm{wd}(B_W(p,p))= \frac{p^2 -1}{4}$ holds. 
\end{prop}

\begin{proof}
When $p$ is an odd number, each pair of strands of $B_W(p,p)$ has exactly two crossings that are both over-crossings of one of the two strands (Proposition 2.1 in \cite{ASA}). 
This means that each strand has either $0$ or $2$ over-crossings for the other strands. 
More precisely, each strand has the same number of strands which are over or under the strand since $B_W(p,p)$ is an alternating diagram. 
Hence, we have the under-crossing set
\begin{align*}
U= \{ \{ 0, 0, \dots , 0, 2, 2, \dots , 2\} , \{ 0, 0, \dots , 0, 2, 2, \dots , 2\} , \dots , \{ 0, 0, \dots , 0, 2, 2, \dots , 2\} \} , 
\end{align*}
which is a multiset of $p$ multisets, and each set is consisting of $(p+1)/2 (=(p-1)/2+1)$ of $0$'s and $(p-1)/2$ of $2$'s. 
This implies that every columns in a OU matrix must have $(p-1)/2$ of $2$'s for any order of strands. 
Therefore, even if one wants to mimimize $f_B$, the first $(p-2)/2$ columns cannot avoid having some $2$'s, more precisely $(p+1)/2-i$ of $2$'s for the $i$th column, below the main diagonal. 
Let $\pi = ( 1, 3, \dots , p, 2, 4, \dots , p-1)$. 
Then the OU matrix $M(\pi)=(m_{ij})$ satisfies
\begin{align*}
m_{ij}= \left\{
\begin{array}{ll}
2 & \left( \ i>j \text{ and } \frac{p+1}{2} \leq i-j \leq p-1, \text{ or } i<j \text{ and } 1 \leq j-i \leq \frac{p-1}{2} \ \right) \\
0 & ( \text{otherwise} )
\end{array}
\right.
\end{align*}
(see Example \ref{ex-w77-m} for $p=7$). 
Observe that $M(\pi)$ realizes the minimum value of $f$, considering each column. 
The warping degree with this order has been calculated in \cite{ASA} to be $(p^2-1)/4$. 
\end{proof}
\medskip

\section{Positive braids}
\label{section-pb}

A {\it positive braid diagram} is a braid diagram with all the crossings positive. 
A {\it positive braid} is a braid which has a positive braid diagram. 
In this section, we investigate the OU matrix of positive braids. 
In Subsection \ref{subsection-inv}, we show that the OU matrix is diagram-independent for positive braids and introduce some invariants for positive braids. 
In Subsections \ref{subsection-perm} and \ref{subsection-fund}, we show examples for positive permutation braids and fundamental braids, respectively.
In Subsection \ref{subsection-pure}, we discuss the OU matrix and the warping degree for positive pure braids.

\subsection{Invariants of positive braids}
\label{subsection-inv}

The following theorem was proved by Garside in \cite{Gp}. 

\medskip 
\begin{thm}[\cite{Gp}]
For any positive braid, any pair of positive braid diagrams are related by a finite sequence of the following moves.
\begin{align*}
\sigma_i \sigma_j \sigma_i = \sigma_j \sigma_i \sigma_j \ & \ (|j-i|=1) \\
\sigma_i \sigma_j = \sigma_j \sigma_i \ & \ (|j-i| \neq 1)
\end{align*}
\end{thm}
\medskip

\noindent Namely, positive braid diagrams of the same positive braid are diagramatically related  only by the Reidemeister moves of type I\hspace{-0.5pt}I\hspace{-0.5pt}I.
The point is that there is no need for the Reidemeister move of type I\hspace{-0.5pt}I. 
Since the OU matrix is not affected by Reidemeister moves of type I\hspace{-0.5pt}I\hspace{-0.5pt}I, we have the following lemma. 

\medskip
\begin{lem}
For any pair of positive braid diagrams $B$ and $B'$ of the same positive braid, $M_B(\pi)=M_{B'}(\pi)$ holds for any $\pi$. 
\label{lem-pd}
\end{lem}
\medskip

\noindent We denote $M_B( \pi )$ for any positive braid diagram $B$ of a positive braid $\beta$ with a permutation $\pi$ by $M_{\beta}( \pi )$. 
We also denote $\mathrm{rank}(B)$, $\mathrm{det}(B)$ by $\mathrm{rank}(\beta)$, $\mathrm{det}( \beta )$, respectively for any positive braid diagram $B$ of a positive braid $\beta$. 
Similarly, we define the over- and under-crossing sets $O(\beta)$ and $U(\beta)$ of $\beta$ to be that of a positive braid diagram $B$ of $\beta$. 
By Lemma \ref{lem-pd} and Corollary \ref{cor-det}, we have the following corollary. 

\medskip 
\begin{cor}
For positive braids $\beta$, $\mathrm{rank}( \beta )$, $\mathrm{det}( \beta )$, $O(\beta)$, $U(\beta)$ are positive braid invariants. 
\label{prop-pbi}
\end{cor}
\medskip 

\begin{ex}
The positive braid $\beta$ illustrated in Figure \ref{fig-det-n} has $\mathrm{rank}(\beta)=3$, $\mathrm{det}(\beta)=2$, and 
\begin{align*} 
O(\beta)= & \{ \{ 0,0,1 \} , \{ 0,1,2 \} , \{ 0,1,2 \} \} \\ 
U(\beta)= & \{ \{ 0,0,2 \} , \{ 0,1,1 \} , \{ 0,1,2 \} \} .
\end{align*}
\begin{figure}[ht]
\centering
\includegraphics[width=3.5cm]{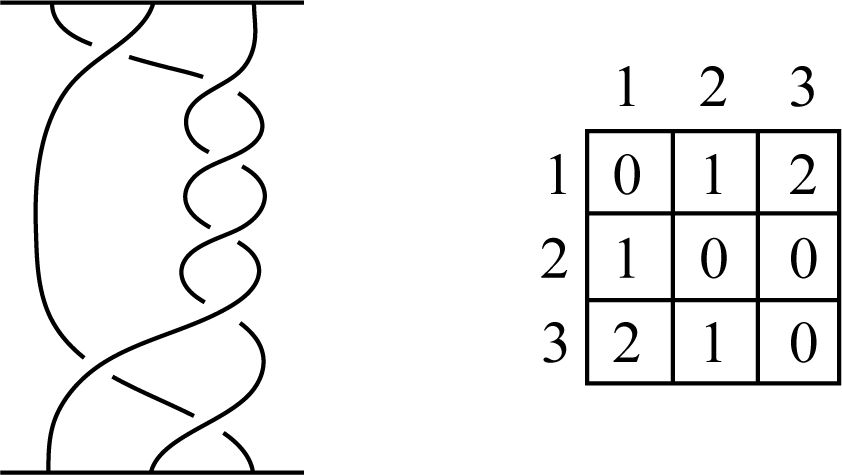}
\caption{A braid diagram $B=\sigma_1 \sigma_2^4 \sigma_1 \sigma_2$ of a positive braid $\beta$.}
\label{fig-det-n}
\end{figure}
\end{ex}

\begin{prop}
For any non-negative integer $n$, there exists a positive braid $\beta$ such that $\mathrm{det}(\beta)=n$.
\end{prop}

\begin{proof}
When $B = \sigma_1 \sigma_2^{2n} \sigma_1 \sigma_2$, we have 
\begin{align*}
M_B(id)=
\begin{pmatrix}
0 & 1 & n \\
1 & 0 & 0 \\
n & 1 & 0
\end{pmatrix}
\end{align*}
and $\mathrm{det}(\sigma_1 \sigma_2^{2n} \sigma_1 \sigma_2)=n$ (see Figure \ref{fig-det-n} for $n=2$). 
\end{proof}

\subsection{Positive permutation braid}
\label{subsection-perm}

A {\it positive permutation braid} is a positive braid such that each pair of strands has at most one crossing. 
There is one to one correspondence from positive permutation $n$-braids to permutations on $(1,2, \dots , n)$ by taking their braid permutations $\rho$, and positive permutation braids have been used in the studies of normal form of braids (see \cite{Artin, Gp, th}). 

\medskip 
\begin{lem}
The OU matrix $M_B (id)$ of a positive permutation braid $\beta$ is a strictly lower triangular matrix in which each component below the main diagonal is either $0$ or $1$. 
\label{lem-l-tri}
\end{lem}

\begin{proof}
When a pair of strands $s_i$ and $s_j$ $(i<j)$ has a crossing, $s_i$ passes under $s_j$, and equivalently $s_j$ passes over $s_i$ at the crossing. 
Hence the $(i,j)$-component of $M_B (id)$ is $0$ and the $(j,i)$-component is either $1$ or $0$ for any pair of $i<j$. 
\end{proof}
\medskip

\noindent We have the following corollary from Lemma \ref{lem-l-tri}. 

\medskip
\begin{cor}
Every positive permutation braid $\beta$ has $\mathrm{det}(\beta)=0$. 
\end{cor}
\medskip 

\noindent For the warping degree, we have the following corollary. 

\begin{cor}
Every positive braid diagram $B$ of a positive permutation braid $\beta$ has $\mathrm{wd}(B)=0$, namely $B$ is completely layered. 
\label{cor-perm-mon}
\end{cor}

\begin{proof}
Let $\pi = ( n, n-1, \dots , 2, 1 )$. 
Then $M_B (\pi)$ is a strictly upper triangular matrix by Lemmas \ref{lem-pi-inv} and \ref{lem-l-tri}. 
Hence, we have the objective function $f_B( \pi )=0$, and therefore $\mathrm{wd}(B)=0$. 
\end{proof}
\medskip

\subsection{Fundamental braid}
\label{subsection-fund}

A {\it fundamental braid}, denoted by $\Delta$, is a positive permutation braid in which each pair of strands has exactly one crossing. 

\medskip
\begin{lem}
The OU matrix $M_{\Delta}(id)$ of a fundamental braid $\Delta$ is a matrix $(m_{ij})$ such that 
\begin{align}
m_{ij}= 
\left\{
\begin{array}{ll}
1 & (i>j) \\
0 & (i \leq j).
\end{array}
\right.
\label{formula-md}
\end{align}
\label{lem-md}
\end{lem}

\begin{proof}
For each pair of strands $s_i$ and $s_j$ $(i<j)$, $s_j$ passes over $s_i$ as $\Delta$ is a positive braid. 
\end{proof}
\medskip 

\noindent We denote the matrix $(m_{ij})$ of (\ref{formula-md}) by $D$ in this section. 
We have the following corollary. 

\medskip 
\begin{cor}
Let $\Delta$ be the fundamental braid of $n$ strands. 
Then, $\mathrm{rank}(\Delta)=n-1$.
\end{cor}
\medskip 

\noindent For braid products, we have the following proposition. 

\medskip 
\begin{prop}
When $r$ is an odd number, $M_{{\Delta}^r}(id)$ is a matrix $(m_{ij})$ such that 
\begin{align*}
m_{ij}= 
\left\{
\begin{array}{cc}
\lfloor \frac{r}{2} \rfloor & (i<j) \\
0 & (i=j) \\
\lceil \frac{r}{2} \rceil & (i>j),
\end{array}
\right.
\end{align*}
namely, $M_{{\Delta}^r}(id)= \lceil \frac{r}{2} \rceil D + \lfloor \frac{r}{2} \rfloor D^T$. 
When $r$ is an even number, $M_{{\Delta}^r}(id)$ is a matrix $( n_{ij})$ such that 
\begin{align*}
n_{ij}= 
\left\{
\begin{array}{cc}
\frac{r}{2}  & (i \neq j) \\
0 & (i=j), 
\end{array}
\right.
\end{align*}
namely, $M_{{\Delta}^r}(id)=\frac{r}{2}D+\frac{r}{2}D^T$. 
\label{prop-dr}
\end{prop}
\medskip 

\noindent To prove Proposition \ref{prop-dr}, we see the braid permutation of $\Delta^r$. 

\medskip
\begin{lem}
When $r$ is an odd number, the braid permutation of $\Delta^r$ is $\rho = (n, n-1, \dots , 2, 1)$. 
When $r$ is an even number, the braid permutation of $\Delta^r$ is $id$. 
\label{lem-dp}
\end{lem}
\medskip

\noindent Now we prove  Proposition \ref{prop-dr}. 

\medskip 
\noindent {\it Proof of Proposition \ref{prop-dr}.} \ 
Let $B$ be a positive braid diagram of $\Delta$. 
When $r$ is an odd number, we prove that $M_{B^r}(id)= \lceil \frac{r}{2} \rceil D + \lfloor \frac{r}{2} \rfloor D^T$ by an induction. 
When $r=1$, we have $M_B(id)=D$ by Lemma \ref{lem-md}. 
Let $\rho =(n, n-1, \dots , 2, 1)$. 
Assume the equality holds for an odd number $k$. 
When $r=k+2$, 
\begin{align*}
M_{B^{k+2}}(id) = M_{B^{k+1} B}(id) = M_{B^{k+1}}(id) + M_{B}(id) 
\end{align*}
by Proposition \ref{prop-product} and Lemma \ref{lem-dp}. This equals
\begin{align*}
& M_{B^k B}(id) + M_B (id) \\
= & M_{B^k}(id) + M_B (\rho ) + M_B (id)\\
= & \left\lceil \frac{k}{2} \right\rceil D + \left\lfloor \frac{k}{2} \right\rfloor D^T + D^T +D \\
= & \left\lceil \frac{k+2}{2} \right\rceil D + \left\lfloor \frac{k+2}{2} \right\rfloor D^T.
\end{align*}
Hence $M_{B^r}(id)= \lceil \frac{r}{2} \rceil D + \lfloor \frac{r}{2} \rfloor D^T$ holds when $r=k+2$, too. 
Therefore $M_{{\Delta}^r}(id)= \lceil \frac{r}{2} \rceil D + \lfloor \frac{r}{2} \rfloor D^T$ holds. 
It similarly holds when $r$ is even. 
\qed

\subsection{Positive pure braid}
\label{subsection-pure}

In this subsection, we show that the OU matrix $M_{\beta}(\pi)$ of a positive pure braid $\beta$ is a symmetric matrix for any strand permutation $\pi$. 

\medskip 
\begin{prop}
The OU matrix $M_{\beta}(\pi)$ of a positive pure braid $\beta$ is a symmetric matrix for any $\pi$. 
\label{lem-symm}
\end{prop}

\begin{proof}
Let $B$ be a positive braid diagram of $\beta$. 
Since $B$ is a pure braid, each pair of strands $s_i$ and $s_j$ has an even number of mutual crossings, and the positions of $s_i$ and $s_j$ are switched at every positive crossing. 
Hence the numbers of over crossings of $s_i$ and $s_j$ are same, and therefore the corresponding $(k,l)$- and $(l,k)$-components of $M_B(\pi)$ have the same value. 
\end{proof}
\medskip 

\noindent As for the warping degree, we have the following corollary from Proposition \ref{lem-symm}.

\medskip 
\begin{cor}
For each positive pure braid diagram $B$, the value of $f_B(\pi)$ does not depend on $\pi$. 
\label{cor-ppb}
\end{cor}
\medskip

\section*{Acknowledgment}

The first author's work was partially supported by JSPS KAKENHI Grant Number JP21K03263. 
The second author's work was partially supported by JSPS KAKENHI Grant Number JP19K03508.

\end{document}